\documentclass[10pt]{amsart} 
\usepackage{amssymb} 
\usepackage{verbatim}
\usepackage{graphicx} 
\usepackage{apn}
\usepackage[all]{xy}
\usepackage{color}

\usepackage[colorlinks=false]{hyperref}
\hypersetup{colorlinks,citecolor=blue,linktocpage}

\numberwithin{equation}{section}

\date{} 
\author{Valery Alexeev and Rita Pardini}
\title{On the existence of ramified abelian covers}

\begin{document}

\begin{abstract}
  Given a normal complete   variety $Y$ over an algebraically closed field $\mathbb K$, distinct irreducible  effective  Weil divisors $D_1, \dots D_n$ of $Y$ and  positive  integers $d_1,\dots d_n$, we spell out the conditions  for the existence of an abelian cover $X\to Y$ branched with order $d_i$ on $D_i$ for $i=1,\dots n$.
  
  As an application, we   prove that a  cover of a  normal complete toric variety 
  branched on  the torus-invariant divisors is itself a toric variety if  $\chr \mathbb K=0$ or if the cover is Galois of degree not divisible by $\chr \mathbb K$.
  \end{abstract}

\maketitle

{\em Dedicated to Alberto  Conte on his $70^{th}$ birthday.}
\setcounter{tocdepth}{1}
\tableofcontents

\section{Introduction}
\label{sec:intro}

Given a  projective variety $Y$ over an algebraically closed field $\mathbb K$ of characteristic $p\ge 0$  and effective divisors $D_1,\dots D_n$ of $Y$,  deciding whether there exists a Galois cover branched on $D_1, \dots D_n$ with given multiplicities is a very  complicated question, which in the complex case is essentially equivalent to describing the finite quotients of the  fundamental group of $Y\setminus(D_1\cup\dots \cup D_n)$. 

In Section \ref{sec:abelian} of this paper we answer  this question
for a normal variety $Y$ in the case that  the Galois group of the
cover is  abelian (Theorem \ref{thm:Gmax}) of order not divisible by $p$, using  the theory developed in 
\cite{Pardini_AbelianCovers} and  \cite{AlexeevPardini_Covers}.
In particular, we prove that when the  class group $\Cl(Y)$ is torsion free, every abelian cover of $Y$ branched on $D_1, \dots D_n$ with given multiplicities not divisible by $p$ is the quotient of a maximal such cover, unique up to isomorphism. 

In Section \ref{sec:toric-covers} we analyze the same question using toric geometry in the case when $Y$ is a normal complete toric variety and $D_1,\dots D_n$ are invariant divisors and we obtain results that parallel those in Section \ref{sec:abelian} (Theorem \ref{thm:Tmax}). Combining the two approaches we are able to show that   any  
 cover of a normal complete toric variety branched on the invariant divisors is toric if  $p=0$ or if the cover is Galois of order not divisible by $p$ (Theorem \ref{thm:abelian-toric}). 
\smallskip

{\em Acknowledgments.} We wish to thank Angelo Vistoli for useful  discussions on the topic of this paper (cf. Remark \ref{rem:angelo}). We also wish to thank  Boaz Moerman, who pointed out to us that the earlier version of Theorem \ref{thm:abelian-toric} was incorrect. 
 This was due to the fact that in the proof we made use of \cite[Prop. 1]{Miyanishi}, where the  assumption that the degree of the cover be prime to $p$  is made only implicitly.
\bigskip

\noindent{\bf Notation.} 
$G$ always denotes a finite  group, almost always abelian,  and $G^*:=\Hom(G,\bK^*)$ the group of characters; $o(g)$ is the order of the element $g\in G$ and $|H|$ is the cardinality of a subgroup $H<G$. We work over an algebraically closed field $\bK$ whose characteristic $p$ does not divide the order of  the finite abelian groups we consider.\\
If $A$ is an abelian group we write $A[d]:=\{a\in A\ |\ da=0\}$ ($d$ an integer),  $A\vi:=\Hom (A,\bZ)$ and we denote by $\Tors(A)$ the torsion subgroup of $A$.\\
The smooth part of a variety $Y$ is denoted by $Y_{\sm}$.
The symbol $\equiv$ denotes linear equivalence of divisors. If $Y$ is a normal variety we denote by $\Cl(Y)$ the group of classes, namely the group of Weil divisors up to linear equivalence. 

\section{Abelian covers}\label{sec:abelian}

\subsection{The fundamental relations} \label{ssec:fundrel}
We quickly recall  the theory of abelian covers (cf. \cite{Pardini_AbelianCovers},  \cite{AlexeevPardini_Covers}, and also \cite{PardiniTovena_fundgroup}) in the most   suitable form  for the applications considered here.

There are slightly different definitions of abelian covers in the literature (see, for instance,  \cite{AlexeevPardini_Covers} that treats also the non-normal case). Here we restrict our attention to the case of normal varieties, but we do not require that the covering map be flat; hence we  define a cover as a finite  morphism $\pi\colon X\to Y$ of normal  
varieties and we say that  $\pi$  is an abelian cover  if it is a Galois morphism  with abelian  Galois group $G$ ($\pi$ is also called a ``$G$-cover'').

Recall that, as already stated in the Notations,   throughout all the paper we assume that $G$ has  order not divisible by $\chr \bK$.

To every component $D$ of the branch locus of $\pi$ we associate the pair $(H,\psi)$, where $H<G$ is the cyclic subgroup consisting of the elements of $G$ that fix the preimage of $D$ pointwise (the {\em inertia subgroup} of  $D$) and $\psi$ is the element  of the character group $H^*$ given by the    natural   representation of $H$ on the normal space to the preimage of $D$ at a general point (these definitions are well posed since $G$ is abelian).
It can be shown  that $\psi$ generates the group $H^*$.

If we fix a primitive  $d$-th root $\zeta$ of $1$, where $d$ is the exponent of the group $G$, then a pair $(H,\psi)$ as above is determined by  the generator    $g\in H$ such that $\psi(g)=\zeta^{\frac {d} {o(g)}}$. We follow  this convention and attach to  every component $D_i$ of the branch locus of $\pi$ a nonzero element $g_i\in G$.

If $\pi$ is flat, which is always the case when $Y$ is smooth, the sheaf $\pi_*\OO_X$ decomposes under the $G$-action as $\oplus_{\chi\in G^*}L_{\chi}\inv$, where the $L_{\chi}$ are line bundles ($L_1=\OO_Y$) and $G$ acts on $L_{\chi}\inv$  via the character $\chi$. 

Given $\chi\in G^*$ and $g\in G$, we denote by $\ol{\chi}(g)$ the smallest non-negative integer $a$ such that $\chi(g)=\zeta^{\frac {ad}{o(g)}}$.  The main result of \cite{Pardini_AbelianCovers} is that the $L_{\chi}$, $D_i$ (the {\em building data} of $\pi$) satisfy the following {\em fundamental relations}:
\begin{equation} \label{eq:fundrel}
L_{\chi}+L_{\chi'}\equiv L_{\chi+\chi'}+\sum_{i=1}^n \epsi^i_{\chi,\chi'} D_i\qquad \forall \chi,\chi' \in G^*
\end{equation}
where $\epsi^i_{\chi,\chi'}=\lfloor \frac{\ol{\chi}(g_i)+\ol{\chi'}(g_i)}{o(g_i)}\rfloor$. (Notice that the coefficients $\epsi^i_{\chi,\chi'}$ are equal either to $0$ or to $1$). Conversely, distinct irreducible divisors $D_i$ and line bundles $L_{\chi}$ satisfying \eqref{eq:fundrel}  are the building data of a flat (normal) $G$-cover  $X\to Y$; in addition, if  $h^0(\OO_Y)=1$ then $X\to Y$ is uniquely determined up to isomorphism of $G$-covers.\\

If we fix characters $\chi_1,\dots \chi_r\in G^*$ such that $G^*$ is the direct sum of the subgroups generated by the $\chi_j$, and we set $L_j:=L_{\chi_j}$, $m_j:=o(\chi_j)$, then the solutions of the fundamental relations \eqref{eq:fundrel} are in one-one correspondence with the solutions of the following {\em reduced fundamental relations}:
\begin{equation} \label{eq:redfund}
m_jL_j\equiv \sum _{i=1}^n\frac{m_j \overline{\chi_j}(g_i)}{d_i}D_i, \qquad j=1,\dots r
\end{equation}

As before, denote by $d$ the exponent of $G$; notice that if  $\Pic(Y)[d]=0$, then for fixed $(D_i, g_i)$, $i=1,\dots n$, the solution of \eqref{eq:redfund} is unique, hence the {\em branch data} $(D_i, g_i)$ determine the cover.
\bigskip

In order to deal with the case when $Y$ is normal but not smooth,  we
observe first that the cover $X\to Y$ can be recovered from  its
restriction $X'\to  Y\sm$ to the smooth locus by taking the integral
closure of $Y$ in the extension  $\bK(X')\supset \bK(Y)$.
Observe then that, since   the complement $Y\setminus Y\sm$ of the smooth part has
codimension $>1$,  we have   $h^0(\OO_{Y\sm})=h^0(\OO_Y)=1$, and thus
  the cover $X'\to Y\sm$ is determined by the building data
$L_{\chi}, D_i$. Using the identification
$\Pic(Y\sm)=\Cla(Y\sm)=\Cla(Y)$, 
we can  regard the $L_{\chi}$ as elements  of $\Cla(Y)$ and, taking the closure, the $D_i$ as Weil divisors on $Y$, and we can interpret the fundamental relations as equalities in $\Cla(Y)$. In this sense, if $Y$ is normal variety with $h^0(\OO_Y)=1$, then the  $G$-covers $X\to Y$ are determined by the building data up to isomorphism.
\bigskip

We say that an abelian cover $\pi\colon X\to Y$ is {\em totally ramified} if  the inertia subgroups of the divisorial components of the branch locus of $\pi$ generate $G$, or, equivalently, if $\pi$ does not factorize through a cover $X'\to Y$ that is \'etale over $Y\sm$. We observe that a totally ramified cover is necessarily connected; conversely, equations \eqref{eq:redfund} imply that if $G$ is an abelian group of exponent $d$ and $Y$ is a variety such that  $\Cla(Y)[d]=0$, then any connected $G$-cover  of $Y$ is totally ramified.

\subsection{The maximal cover}\label{ssec:maxcover}

Let $Y$ be a complete normal   variety,  let $D_1, \dots D_n$ be distinct  irreducible effective divisors of $Y$ and let $d_1,\dots d_n$ be positive integers (it is convenient to allow the possibility that $d_i=1$ for  some $i$). We set $d:={\rm lcm}(d_1,\dots d_n)$.

We say that a Galois cover $\pi\colon X\to Y$ is {\em branched on $D_1, \dots D_n$ with orders $d_1,\dots d_n$} if:

\begin{itemize}
\item the divisorial part of the branch  locus of $\pi$ is contained in $\sum_i D_i$; 
\item the ramification order of $\pi$ over $D_i$ is equal to $d_i$.
\end{itemize}

Let $\eta\colon \wt Y\to Y$ be a resolution of the singularities and
set  
${\NN}(Y):=\Cla(Y)/\eta_*\!\Pic^0(\wt Y)$. Since  the map $\eta_*\colon \Pic(\wt Y)=\Cla(\wt Y)\to \Cl(Y)$ is surjective, $\N(Y)$ is a quotient of  the N\'eron-Severi group $\NS(\wt Y)$, hence it is finitely generated.  It follows that $\eta_*\!\Pic^0(\wt Y)$ is the largest divisible subgroup of $\Cla(Y)$ and therefore $\N(Y)$ does not depend on the choice of the resolution of $Y$ (this is easily checked also by a geometrical argument).  
The group $\Cla(Y)\vi$ coincides with $\N(Y)\vi$, hence it is a finitely generated free abelian group of rank equal to the rank of $\N(Y)$.

Consider the map  $\bZ^n\to \Cl(Y)$ that maps the $i$-th canonical generator to the class of $D_i$,
  let $\phi\colon \Cla(Y)^{\vee}\to \oplus_{i=1}^n\bZ_{d_i}$  be the map obtained by composing the dual map  $\Cla(Y)\vi\to (\bZ^n)\vi$ with $(\bZ^n)\vi =\bZ^n\to \oplus_{i=1}^n\bZ_{d_i}$ and let $K_{\min}$  be the image of $\phi$. 
  Let $G_{\max}$ be the abelian group defined by the exact sequence:
\begin{equation}\label{eq:Gmax}
0\to K_{\min} \to \oplus_{i=1}^n \bZ_{d_i}\to G_{\max}\to 0.
\end{equation} 

Then we have the following:

\begin{thm} \label{thm:Gmax}
Let $Y$ be a normal    variety with $h^0(\OO_Y)=1$, let $D_1, \dots D_n$ be distinct irreducible effective divisors, let $d_1, \dots d_n$ be positive integers and set $d:={\rm lcm}(d_1,\dots d_n)$. 
Then:
\begin{enumerate}
\item If $X\to Y$ is a totally ramified $G$-cover branched on $D_1,\dots D_n$ with orders $d_1,\dots d_n$, then:
\begin{itemize}
\item[(a)] the map $\oplus_{i=1}^n\Z_{d_i}\to G$ that maps $1\in \bZ_{d_i}$ to $g_i$  descends to  a surjection $G_{\max}\twoheadrightarrow G$;
\item[(b)] the map $\bZ_{d_i}\to G_{\max}$ is injective for every $i=1,\dots n$.
\end{itemize}
\item If  the map $\bZ_{d_i}\to G_{\max}$  is injective for $i=1,\dots n$ and $\NN(Y)[d]=0$, then there exists a maximal  totally ramified abelian cover  $X_{\max}\to Y$  branched on $D_1,\dots D_n$ with orders $d_1, \dots d_n$; the  Galois group of $X_{\max}\to Y$ is equal to $G_{\max}$.

\item If  the map $\bZ_{d_i}\to G_{\max}$  is injective for $i=1,\dots n$ and $\Cla(Y)[d]=0$,  then the cover $X_{\max}\to Y$ is unique up to isomorphism of $G_{\max}$-covers and every totally ramified abelian cover $X\to Y$ branched on $D_1,\dots D_n$ with orders $d_1,\dots d_n$ is a quotient of $X_{\max}$ by a subgroup of $G_{\max}$.
\end{enumerate}
\end{thm}
\begin{proof} Let $H_1,\dots H_t\in \NN(Y)$ be  elements whose classes are free generators of the abelian group $\NN(Y)/\Tors(\NN(Y))$, and write:
\begin{equation}\label{eq:Hj}
D_i=\sum_{j=1}^t a_{ij}H_j \mod \Tors(\NN(Y)), \qquad j=1,\dots t
\end{equation}
Hence, the subgroup $K_{\min}$ of $\oplus_{i=1}^n \bZ_{d_i}$ is generated by the elements $z_j:=(a_{1j},\dots a_{nj})$, for $j=1,\dots t$. 

Let $X\to Y$ be a totally ramified $G$-cover branched on $D_1,\dots D_n$ with orders $d_1,\dots d_n$ and let $(D_i, g_i)$ be its  branch data. Consider the map $\oplus_{i=1}^n \bZ_{d_i}\to G$ that maps $1\in \bZ_{d_i}$ to $g_i$: this map is surjective, by the assumption that $X\to Y$ is totally ramified, and its restriction to $\bZ_{d_i}$ is injective for $i=1,\dots n$, since the cover is branched  on $D_i$ with order $d_i$. If we denote by $K$ the kernel of $\oplus_{i=1}^n\bZ_{d_i}\to G$, to prove (1) it suffices to show that $K\supseteq K_{\min}$. Dually, this is equivalent to showing that $G^*\subseteq K_{\min}^{\perp}\subset \oplus_{i=1}^n(\bZ_{d_i})^*$. Let $\psi_i\in (\bZ_{d_i})^*$  be the generator that maps $1\in \bZ_{d_i}$ to $\zeta^{\frac {d}{d_i}}$ and write $\chi\in G^*$ as $(\psi_1^{b_1}, \dots \psi_n^{b_n})$, with $0\le b_i< d_i$; if $o(\chi)=m$ then \eqref{eq:redfund} gives $mL_{\chi}\equiv \sum_{i=1}^n \frac {mb_i}{d_i}D_i $. Plugging  \eqref{eq:Hj} in this equation  we obtain that  $\sum_{i=1}^n \frac {b_ia_{ij}}{d_i}$ is an integer for  $j=1,\dots t$, namely $\chi\in K_{\min}^{\perp}$.
\smallskip 

(2) Let  $\chi_1, \dots \chi_r$ be a basis of $G_{\max}^*$ and, as above, for $s=1,\dots r$ write $\chi_s=(\psi_1^{b_{s1}}, \dots \psi_n^{b_{sn}})$, with $0\le b_{si}< d_i$. Since by assumption $\NN(Y)[d]=0$,  by the proof of (1) the elements $\sum_{j=1}^t(\sum_{i=1}^n\frac{ b_{si}a_{ij}}{d_i})H_j$,  $s=1, \dots r$, can be lifted to   solutions $\ol{L_s}\in \NN(Y)$ of the reduced fundamental relations \eqref{eq:redfund}   for a $G_{\max}$-cover with branch data $(D_i, g_i)$, where $g_i\in G$ is the image of $1\in \bZ_{d_i}$. Since the kernel  of $\Cla(Y)\to N(Y)$ is a divisible group,   it is possible to lift the $\ol{L_s}$ to solutions $L_s\in \Cla(Y)$. We let $X_{\max}\to Y$ be the $G_{\max}$-cover determined by these solutions. It is a totally ramified cover since the map $\oplus_{i=1}^n \bZ_{d_i}\to G_{\max}$ is surjective by the definition of $G_{\max}$.
\smallskip

(3) Since $\Cla(Y)[d]=0$,  any $G$-cover such that the exponent of $G$ is a divisor of $d$ is determined uniquely by the branch data; in particular, this holds for the cover $X_{\max}\to Y$ in (2)  and for every intermediate cover $X_{\max}/H\to Y$, where $H<G_{\max}$.  The claim now follows by (1). 
\end{proof}
\begin{example} Take $Y=\pp^{n-1}$ and let $D_1, \dots D_n$ be the coordinate hyperplanes. In this case  the group $K_{\min}$ is generated by $(1,\dots 1)\in \oplus_{i=1}^n \bZ_{d_i}$. Since any connected  abelian cover of $\pp^{n-1}$ is totally ramified,   by Theorem \ref{thm:Gmax} there exists a abelian cover of $\pp^{n-1}$ branched over $D_1, \dots D_n$ with orders   $d_1,\dots d_n$ iff $d_i$ divides $lcm(d_1,\dots ,\widehat{d_i}, \dots d_n)$  for every  $i=1,\dots n$. 
For $d_1=\dots =d_n=d$, then $G_{\max}=\Z_d^{n}/\!\!<\!(1,\dots 1)\!>$ and $X_{\max}\to \pp^{n-1}$ is the cover $\pp^{n-1}\to\pp^{n-1}$ defined by $[x_1,\dots x_n]\mapsto [x_1^d,\dots x_n^d]$.

 In general, $X_{\max}$ is a weighted projective space $\pp(\frac {d}{ d_1},\dots \frac {d} {d_n})$ and the cover is given by $[x_1,\dots x_n]\mapsto [x_1^{d_1}, \dots x_n^{d_n}]$. 
\end{example}

\section{Toric covers}
\label{sec:toric-covers}

\begin{notations}\label{notations:toric}
  Here, we fix the notations which are standard in toric geometry. A  (complete normal) 
  toric variety $Y$ corresponds to a fan $\Sigma$ living in the vector
  space $N\otimes\bR$, where $N\cong\bZ^s$. The dual lattice is
  $M=N^{\vee}$. The torus is $T=N\otimes\bC^* = \Hom(M,\bC^*)$. 

  The integral vectors $r_i\in N$ will denote the integral generators
  of the rays $\sigma_i\in\Sigma(1)$ of the fan $\Sigma$. They are in
  a bijection with the $T$-invariant Weil divisors $D_i$
  ($i=1,\dotsc,n$) on $Y$.
\end{notations}

\begin{definition}\label{def:toric-cover}
  A \emph{ toric cover}
   $f\colon X\to Y$ is a finite morphism of
  toric varieties corresponding to the map of fans $F\colon
  (N',\Sigma')\to (N,\Sigma)$ such that:
  \begin{enumerate}
  \item $N'\subseteq N$ is a sublattice of finite index, so that
    $N'\otimes\bR = N\otimes\bR$.
  \item $\Sigma'=\Sigma$.
  \end{enumerate}
\end{definition}

The proof of the following lemma is immediate.

\begin{lemma}\label{lem:ab-toric-cover}
  The morphism $f$ has the following properties:
  \begin{enumerate}
  \item It is equivariant with respect to the
    homomorphism of tori $T'\to T$.
  \item It is an abelian  cover with
   Galois group $G = \ker(T'\to T) = N/N'$.
  \item It is ramified only along the boundary divisors
    $D_i$, with multiplicities $d_i\ge1$ defined by the condition that
    the integral generator of $N'\cap \bR_{\ge0} r_i$ is $d_ir_i$.

  \end{enumerate}

\end{lemma}

\begin{proposition}\label{prop:dia-from-cover}
  Let $Y$ be a complete  toric variety such that $\Cla(Y)$ is torsion free, and $X\to
  Y$ be a toric cover. Then, with notations as above, there
  exists the following commutative diagram with exact rows and columns.
  \begin{displaymath}
    \xymatrix{
      &&0 \ar[d]
      &0 \ar[d] \\
      &&\Cla(Y)^{\vee} \ar[r]\ar[d]
      &K \ar[d] \\
      0 \ar[r] 
      &\oplus_{i=1}^n \bZ d_i D_i^* \ar[r]\ar^{p'}[d]
      &\oplus_{i=1}^n \bZ D_i^* \ar[r]\ar^p[d]
      &\oplus_{i=1}^n \bZ_{d_i}  \ar[r]\ar[d]
      &0 \\
      0 \ar[r]
      &N' \ar[r]
      &N \ar[r]\ar[d]
      &G \ar[r]\ar[d]
      &0 \\
      &&0
      &0 
      }
  \end{displaymath}
  (Here the $D_i^*$ are formal symbols denoting a basis of $\bZ^n$).
 Moreover, each of the homomorphisms
  $\bZ_{d_i} \to G$ is an embedding.
\end{proposition}
\begin{proof}
  The third row appeared in Lemma \ref{lem:ab-toric-cover}, and the second row is the
  obvious one. 

  It is well known that the boundary divisors on a complete
  normal  toric variety span the group $\Cla(Y)$,  and that there
  exists the following  short exact sequence of lattices:
  $$0 \to  M \lra \oplus_{i=1}^n \bZ D_i \lra \Cla(Y) \to 0.$$
  Since $\Cla(Y)$ is torsion free by assumption, this sequence is split and dualizing it one obtains the central column. 
  Since $\oplus_{i=1}^n\bZ D_i^* \to N$ is surjective, then so is
  $\oplus_{i=1}^n\bZ_{d_i}  \to G$. The group $K$ is defined as the 
  kernel of this map.

  Finally, the condition that $\bZ_{d_i} \to G$ is injective is
  equivalent to the condition that the integral generator of $N'\cap
  \bR_{\ge0} r_i$ is $d_ir_i$, which holds by Lemma~\ref{lem:ab-toric-cover}.
\end{proof}

\begin{theorem} \label{thm:Tmax}
Let $Y$ be a complete  toric variety such that $\Cla(Y)$ is torsion free, let $d_1,\dots d_n$ be positive integers and 
let $K_{\min}$ and $G_{\max}$ be defined as in sequence \eqref{eq:Gmax}. Then:
\begin{enumerate}
\item There exists a toric cover branched on $D_i$ of order $d_i$, $i=1,\dots n$,  iff the map $\bZ_{d_i}\to G_{\max}$ is injective for $i=1,\dots n$. 
\item If condition (1) is satisfied, then among all the  toric covers of $Y$ ramified over the
  divisors~$D_i$  with multiplicities $d_i$ there exists a maximal
  one $X_{\tmax}\to Y$, with Galois group $G_{\max}$,  such that 
  any other  toric cover $X\to Y$ with the same
  branching  orders is a quotient $X=X_{\tmax}/H$ by a
  subgroup $H< G_{\max}$.
  \end{enumerate}
\end{theorem}

\begin{proof}
 Let $X\to Y$ be a toric cover branched on $D_1,\dots D_n$ with orders $d_1,\dots d_n$, let  $N'$ be the corresponding  sublattice of $N$ and $G=N/N'$ the Galois group. 
  Let $N'_{\min}$ be the subgroup of $N$ generated by $d_ir_i$, $i=1,\dots n$. By Lemma \ref{lem:ab-toric-cover} 
  one must have $N'_{\min} \subseteq N'$, hence the map $\Z_{d_i}\to N/N'_{\min}$ is injective since $\Z_{d_i}\to G=N/N'$ is injective by Proposition \ref{prop:dia-from-cover}.  We set $X_{\tmax}\to Y$ to be the cover for
  $N'_{\min}$. Clearly, the cover for the
  lattice $N'$ is a quotient of the cover for the lattice $N'_{\min}$
  by the group $H=N'/N'_{\min}$.

  Consider the second and third rows of the diagram of Proposition \ref{prop:dia-from-cover} as a short exact sequence of
  2-step complexes $0\to A^{\bullet} \to B^{\bullet} \to C^{\bullet}
  \to 0$. The associated long exact sequence of cohomologies gives
  \begin{displaymath}
    \xymatrix{
      \Cla(Y)^{\vee} \ar[r]
      &K \ar[r]
      &\coker(p') \ar[r]
      &0
      }      
  \end{displaymath}
For  $N'=N'_{\min}$, the map $p'$ is surjective, hence  $\Cla(Y)\vi\to K$ is surjective too, and  $K=K_{\min}$, $N/N'_{\min}=G_{\max}$. 

Vice versa, 
 suppose that in the following commutative diagram
  with exact row and columns  each of the maps $\bZ_{d_i}\to
  G_{\max}$ is injective. 
  \begin{displaymath}
    \xymatrix{
      &&0 \ar[d]
      &0 \ar[d] \\
      &&\Cla(Y)^{\vee}\ar[r]^q\ar[d]
      &K_{\min}\ar[d]\ar[r]
      &0 \\
      0 \ar[r] 
      &\oplus_{i=1}^n \bZ d_i D_i^* \ar[r]
      &\oplus_{i=1}^n \bZ D_i^* \ar[r]\ar^p[d]
      &\oplus_{i=1}^n \bZ_{d_i}  \ar[r]\ar[d]
      &0 \\
      &
      &N \ar[d]
      &G_{\max} \ar[d]
      & \\
      &&0
      &0 
      }
  \end{displaymath}
   We complete the first  row on the left by adding $\ker(q)$. We have
  an induced homomorphism $\ker(q) \to \oplus\bZ d_iD_i^*$, and we 
  define $N'$ to be its cokernel. 

  Now consider the  completed first and  second rows as a short
  exact sequence of 2-step complexes $0\to A^{\bullet} \to B^{\bullet}
  \to C^{\bullet} \to 0$. The associated long exact sequence of
  cohomologies says that $\ker(q) \to \oplus_{i=1}^n\bZ d_iD_i^*$ is
  injective, and the sequence
  \begin{displaymath}
    \xymatrix{
      0 \ar[r]
      &N' \ar[r]
      &N \ar[r]
      &G_{\max} \ar[r]
      &0 
    }
  \end{displaymath}
  is exact. It follows that $N'=N'_{\min}$  and the toric morphism 
  $(N'_{\min},\Sigma)\to (N,\Sigma)$  is then the searched-for
  maximal abelian toric cover.
\end{proof}
\begin{remark} 
Condition (1) in the statement of Theorem \ref{thm:Tmax} can also be expressed by saying that for  $i = 1,...,n$ the element $d_ir_i\in N'_{\min}$  is primitive, where $N'_{\min}\subseteq N$ is the subgroup generated by all  the  $d_ir_i$. 
\end{remark}

We now combine the results of this section with those of \S \ref{sec:abelian} to obtain a structure result for Galois  covers of toric varieties. 
\begin{theorem}\label{thm:abelian-toric}
Let $Y$ be a normal complete toric variety and let   $f\colon X\to Y$  be a connected 
cover such that  the divisorial part of the branch locus of $f$ is contained in the union of the invariant divisors $D_1,\dots D_n$.\\

If $\chr \bK=0$ or $f$ is Galois and  $\chr\bK$ does not divide $\deg f$,  then     $f\colon X\to Y$ is a toric cover.
\end{theorem}
\begin{proof}
Our first step is showing that $f$ is an abelian cover. Let $U\subset Y$ be the open orbit and let $X'\to U$ be   the cover obtained by restricting $f$. Since $U$ is smooth, by  the assumptions and by purity of the branch locus, $X'\to U$ is \'etale. If $f$ is Galois, then  $X'\to U$  is also Galois with the same Galois group,  and if  in addition  $\chr\bK$ does  not divide  $\deg f$ then the Galois group is abelian by   \cite[Thm.~1]{Miyanishi}  (cf. also \cite{brion-szamuely}). 

Now drop the assumption that $f$ is Galois but assume  that $\chr\bK=0$.  
Let $X''\to U$ be the Galois closure of $X'\to U$:  the cover $X''\to U$ is also \'etale, hence again 
by \cite[Thm.~1]{Miyanishi}  it is, up to isomorphism, a homomorphism of tori.
It follows that  $X''\to U$ is an abelian cover and so  the intermediate cover $X'\to U$ is also abelian (actually $X'=X''$). The cover $f\colon X\to Y$ is   abelian, too,  since  $X$ is the integral closure of $Y$ in $\bK(X')$. 
So we have proven that under our assumptions $f$ is an abelian cover. We denote by $G$ the Galois group of $f$ and by   $d_1,\dots d_n$  the orders of ramification of $X\to Y$ on $D_1,\dots D_n$.\\
Assume first that  $\Cla(Y)$ has no torsion, so that  every connected  abelian cover of $Y$ is totally ramified (cf. \S \ref{sec:abelian}). 
Then  by  Theorem \ref{thm:Gmax} every connected  abelian cover  branched on $D_1,\dots D_n$ with orders $d_1,\dots d_n$ is a quotient of the maximal abelian cover $X_{\max}\to Y$ by a subgroup $H<G_{\max}$. In particular, this is true for the cover $X_{\tmax}\to Y$ of Theorem \ref{thm:Tmax}. Since $X_{\max}$ and $X_{\tmax}$ have the same Galois group it follows that $X_{\max}=X_{\tmax}$. Hence $X\to Y$, being a quotient of $X_{\tmax}$, is a toric cover. 

Consider now the general case. Recall that the group $\Tors\Cl(Y)$ is finite, 
isomorphic to $N/ \langle r_i \rangle$, and  the cover $Y'\to Y$
corresponding to $\Tors\Cl(Y)$ is toric, and one has
$\Tors\Cl(Y')=0$.
Indeed, on a toric variety the group $\Cl(Y)$ is generated by the $T$-invariant
Weil divisors $D_i$. Thus, $\Cl(Y)$ is the quotient of the free
abelian group $\oplus \bZ D_i$ of all $T$-invariant divisors modulo
the subgroup $M$ of principal $T$-invariant divisors. Thus,
$\Tors\Cl(Y)\simeq M'/M$, where $M'\subset \oplus\bQ D_i$ is the
subgroup of $\bQ$-linear functions on $N$ taking integral values on
the vectors $r_i$. Then $N':=M'^\vee$
is the subgroup of $N$ generated by the $r_i$, and the cover $Y'\to Y$ is
the cover corresponding to the map of fans $(N',\Sigma)\to
(N,\Sigma)$. On $Y'$ one has $N'=\langle r_i\rangle$, so $\Tors\Cl(Y')=0$. 

Let $X'\to Y'$ be a connected component of
the pull back of $X\to Y$: it is an abelian cover branched on the invariant divisors of $Y'$, hence by the first part of the proof it is toric. The  map $X'\to Y$  is toric, since it is a composition of toric morphisms, hence the intermediate cover $X\to Y$ is also toric. 
\end{proof}
\begin{remark}\label{rem:angelo}
The argument that shows that the map $f$   is an abelian cover  in the proof of Theorem \ref{thm:abelian-toric} was suggested  to us by  Angelo Vistoli. He also remarked  that it is possible to prove  Theorem \ref{thm:abelian-toric} in a more conceptual way by showing that the torus action on the  cover  $X'\to U$ of the open orbit of $Y$ extends to $X$, in view of the properties of the integral closure. However our approach has the advantage of describing explicitly the fan/building data associated with the cover.
\end{remark}


\begin{thebibliography}{ABCD}
\bibitem[AP12]{AlexeevPardini_Covers} V.~Alexeev,  R.~Pardini,
  \emph{Non-normal abelian covers}, Compositio Math. \textbf{148}
  (2012), no.~04, 1051--1084.
  \bibitem[BS13]{brion-szamuely} M.~Brion, T.~Szamuely, {\em Prime-to-$(p)$ {\'e}tale covers of algebraic groups and homogeneous spaces},
 Bull. Lond. Math. Soc., {\bf 45} (3) (2013) 602--612. 
\bibitem[Mi72]{Miyanishi} M.~Miyanishi, {\em On the algebraic fundamental group of an algebraic group}, J. 
Math. Kyoto Univ. 12 (1972), 351--367. 

\bibitem[Par91]{Pardini_AbelianCovers}
 R.~Pardini, \emph{{A}belian
    covers of algebraic varieties}, J. Reine Angew.
  Math. \textbf{417} (1991), 191--213. 
  
  \bibitem[PT95]{PardiniTovena_fundgroup}
  R.~Pardini, F.~Tovena, {\em On the fundamental group of an abelian cover},   International  J.  Math.  {\bf 6} no. 5 (1995), 767--789. 

\end{thebibliography}
\end{document}